\documentclass[11pt,twoside,reqno]{amsart}

\usepackage{microtype}
\usepackage[OT1]{fontenc}
\usepackage{type1cm}
\usepackage{amssymb}
\usepackage{enumerate}
\usepackage{xcolor}
\usepackage{mathrsfs}
\usepackage{dsfont}
\usepackage{tikz}
\usetikzlibrary{calc,patterns,angles,quotes,arrows}
\usepackage{verbatim}
\usepackage{cite}
\usepackage{dsfont}
\usepackage{soul}

\usepackage[left=2.3cm,top=3cm,right=2.3cm]{geometry}
\numberwithin{equation}{section}

\theoremstyle{plain}
\newtheorem{theorem}{Theorem}[section]

\newtheorem{proposition}[theorem]{Proposition}

\newtheorem{lemma}[theorem]{Lemma}

\newtheorem*{dimension-drop}{Dimension drop conjecture}

\theoremstyle{remark}

\newtheorem*{ack}{Acknowledgement}

\theoremstyle{definition}

\newcommand{\DD}{\mathcal{D}}

\newcommand{\EE}{\mathcal{E}}
\newcommand{\CC}{\mathcal{C}}

\newcommand{\R}{\mathbb{R}}

\newcommand{\N}{\mathbb{N}}

\newcommand{\iii}{\mathtt{i}}
\newcommand{\jjj}{\mathtt{j}}
\newcommand{\kkk}{\mathtt{k}}
\renewcommand{\lll}{\mathtt{l}}
\newcommand{\eps}{\varepsilon}
\newcommand{\fii}{\varphi}

\newcommand{\dd}{\,\mathrm{d}}

\renewcommand{\ge}{\geqslant}
\renewcommand{\le}{\leqslant}
\renewcommand{\geq}{\geqslant}
\renewcommand{\leq}{\leqslant}

\DeclareMathOperator{\dimh}{dim_H}

\DeclareMathOperator{\dimp}{dim_p}

\DeclareMathOperator{\dims}{dim_{sim}}

\DeclareMathOperator{\diam}{diam}
\DeclareMathOperator{\proj}{proj}

\DeclareMathOperator{\conv}{conv}

\begin{document}

\title{Super-exponential condensation without exact overlaps}

\author{Bal\'azs B\'ar\'any}
\address[Bal\'azs B\'ar\'any]
        {Budapest University of Technology and Economics \\
         Department of Stochastics\\
         MTA-BME Stochastics Research Group\\
         P.O.\ Box 91\\
         1521 Budapest\\
         Hungary}
\email{balubsheep@gmail.com}

\author{Antti K\"aenm\"aki}
\address[Antti K\"aenm\"aki]
        {Department of Physics and Mathematics \\
         University of Eastern Finland \\
         P.O.\ Box 111 \\
         FI-80101 Joensuu \\
         Finland}
\email{antti.kaenmaki@uef.fi}

\thanks{The research of B\'ar\'any was supported by the grants OTKA K123782, NKFI PD123970, and the J\'anos Bolyai Research Scholarship of the Hungarian Academy of Sciences.}
\date{\today}

\begin{abstract}
  We exhibit self-similar sets on the line which are not exponentially separated and do not generate any exact overlaps. Our result shows that the exponential separation, introduced by Hochman in his groundbreaking theorem on the dimension of self-similar sets, is too weak to describe the full theory.
\end{abstract}

\maketitle


\section{Introduction}

A self-similar set consists of similar copies of itself. A classical result of Hutchinson \cite{Hutchinson1981} shows that if these copies are separated enough, then the Hausdorff dimension of the self-similar set equals the similarity dimension, a natural upper bound for the dimension. In order to handle overlaps, Simon and Pollicott \cite{PoSi} introduced the transversality condition. Simon and Solomyak \cite{SiSo} used this condition to show that in the line, for almost every choice of translations, the dimension of the self-similar set equals the similarity dimension.

In his seminal paper, Hochman \cite{Hochman2014} strengthened the estimates on the exceptional parameters for which the dimension drops below the similarity dimension. He showed that exponential separation suffices for the equality of the Hausdorff and similarity dimensions. The classical transversality argument gives an upper bound (which depends on the similarity dimension) for the dimension of the exceptional parameters; see Peres and Schlag~\cite{PeSch}. Hochman managed to show that the packing dimension of the exceptional set is zero. Furthermore, Shmerkin and Solomyak \cite{ShmerkinSolomyak2016} used similar techniques and conditions to study the absolute continuity of self-similar measures, and Shmerkin \cite{Shmerk} applied this approach to study the $L^q$-spectrum of self-similar measures.

A folklore conjecture proposes that the only possibility for the Hausdorff dimension to be strictly less than the similarity dimension is the existence of exact overlaps. The conjecture can analogously be formulated for self-similar measures. Varj\'u \cite{Varju} studied the dimension of Bernoulli convolutions, which is a certain class of self-similar measures. He proved that there is no dimension drop if the contraction parameter is transcendental. A corollary of Hochman's result \cite{Hochman2014}, which is stated in Breuillard and Varj\'u~\cite{BV}, implies that if the contraction ratio is algebraic, then the dimension of the Bernoulli convolution can be explicitly calculated. In particular, this means that the conjecture holds for Bernoulli convolutions. 
	
Hochman \cite{Hochman2014} showed that the dimension drop implies super-exponential condensation. In \cite[p.~1948]{Hochmanproc}, he remarked that it is not known if there exists a super-exponentially condensated self-similar set without exact overlaps. He also speculated that such self-similar sets simply do not exist, which would then prove the conjecture. We answer this in negative by constructing uncountably many parametrized homogeneous self-similar sets having super-exponential condensation but no exact overlaps. Very recently, independently of us, Baker \cite{Baker2019} showed the existence of such a self-similar set. In fact, after his result appeared online, we decided to make our considerations public as well. While Baker applied the theory of continued fractions, our proof relies on non-linear projections and the transversality condition.


The observation that the super-exponential condensation does not imply the exact overlapping means that, in order to verify or disprove the conjecture, one has to study the overlaps in a more sophisticated way. By applying Hochman \cite{Hochman2014}, we characterize the dimension drop of the natural measure on a homogeneous self-similar set by means of the average exponential separation. Our results, therefore, introduce a possible roadmap to disprove the conjecture.

\section{Preliminaries and main results}

We consider a tuple $\Phi = (\fii_i)_{i \in I}$, where $I$ is a finite index set, of contracting similitudes acting on $\R^d$. Each of the map $\fii_i$ has the form $\fii_i(x) = \lambda_iO_ix+t_i$, where $0<\lambda_i<1$ is the \emph{contraction}, $O_i$ the \emph{orthogonal part}, and $t_i \in \R^d$ the \emph{translation} of $\fii_i$. We say that $\Phi$ is \emph{homogeneous} if there exists $0<\lambda<1$ such that $\lambda_i=\lambda$ for all $i \in I$. A \emph{self-similar set} associated to $\Phi$ is the unique non-empty compact set $X \subset \R^d$ for which
\begin{equation} \label{eq:ss-invariant}
  X = \bigcup_{i \in I} \fii_i(X).
\end{equation}
The existence and uniqueness of such sets was proved by Hutchinson \cite{Hutchinson1981}.
The self-similar set $X$ is \emph{homogeneous} if it is given by a homogeneous tuple. Writing $\fii_\iii = \fii_{i_1} \circ \cdots \circ \fii_{i_n}$ and $\lambda_\iii = \lambda_{i_1} \cdots \lambda_{i_n}$, we have $\diam(\fii_\iii(B)) = \lambda_\iii \diam(B) \le (\max_{i \in I} \lambda_i)^n \diam(B)$ for all sequences $\iii = i_1 \cdots i_n \in I^n$ and sets $B \subset \R^d$. Therefore, defining $\iii|_n = i_1 \dots i_n$ for all $\iii = i_1i_2\cdots \in I^\N$, we see that $\diam(\fii_{\iii|_n}(B)) \to 0$ as $n \to \infty$ for all $\iii \in I^\N$ and bounded sets $B \subset \R^d$. Each $\iii \in I^\N$ corresponds to one point in $X$ via the \emph{canonical projection} $\pi$ defined by the relation
\begin{equation*}
  \{\pi(\iii)\} = \{\lim_{n \to \infty} \fii_{\iii|_n}(0)\} = \bigcap_{n=1}^\infty \fii_{\iii|_n}(B(0,R)),
\end{equation*}
where $R = \max_{i \in I}|\fii_i(0)|/(1-\max_{i \in I}\lambda_i)$ and $\fii_i(B(0,R)) \subset B(0,R)$ for every $i\in I$. In fact, it is easy to see that $\pi(I^\N) = X$ and hence, the canonical projection introduces an alternative way to define the self-similar set. By iterating \eqref{eq:ss-invariant}, we see that $X = \bigcup_{\iii \in I^n} \fii_\iii(X)$ for all $n \in \N$. Therefore, the family $\{ \fii_\iii(B(0,R)) \}_{\iii \in I^n}$ consisting of balls as small as we wish is a natural cover for $X$. It is easy to see that $\dimh(X) \le \dims(\Phi)$, where $\dimh$ is the Hausdorff dimension and the \emph{similarity dimension} $\dims(\Phi)$ is the unique number $s \ge 0$ for which $\lim_{n \to \infty}(\sum_{\iii \in I^n} \lambda_\iii^sR^s)^{1/n} = \sum_{i \in I} \lambda_i^s = 1$.

It is well known that if the \emph{strong separation condition} is satisfied, which means that $\fii_i(X) \cap \fii_j(X) = \emptyset$ whenever $i \ne j$, then $\dimh(X) = \dims(\Phi)$. The strong separation condition can be relaxed to a slightly weaker assumption, called the open set condition, which, roughly speaking, means that the overlapping of the sets $\fii_\iii(X)$ of essentially the same diameter has bounded multiplicity. It has to be emphasized that the open set condition only allows ``slight overlaps''. For example, if $X$ has \emph{exact overlaps}, meaning that there are finite sequences $\iii \ne \jjj$ such that $\fii_\iii = \fii_\jjj$, then $\dimh(X) < \dims(\Phi)$. Indeed, by denoting the \emph{length} of $\iii$ by $|\iii|$ and the \emph{concatenation} of $\iii$ and $\jjj$ by $\iii\jjj$, we may, by replacing $\iii$ and $\jjj$ by $\iii\jjj$ and $\jjj\iii$, assume that the finite sequences $\iii$ and $\jjj$ have the same length $|\iii|=|\jjj|=n$. Therefore, $\dimh(X) \le \dims(\Phi^n) < \dims(\Phi)$, where $\Phi^n$ is the tuple consisting of $(\# I)^n - 1$ many $n$-length compositions of the maps $\fii_i$ -- all of them but $\fii_\jjj$. Currently, for self-similar sets in the real line with $\dims(\Phi) \le 1$, no other mechanism is known which drops the dimension of $X$ below the similarity dimension. Trivially, if $\dims(\Phi)>1$, then $\dimh(X) \le 1$. The following folklore conjecture has probably first time been stated by Simon \cite{Siex}.

\begin{dimension-drop}
  If $\Phi$ is a tuple of contractive similitudes acting on the real line and $X \subset \R$ is the associated self-similar set such that $\dimh(X) < \min\{1,\dims(\Phi)\}$, then $X$ has exact overlaps.
\end{dimension-drop}

There exist a version of the conjecture also in higher dimensions, see Hochman \cite[Conjecture 1.3]{Hochmanrd}, but from now on, unless otherwise stated, we work only on the real line. In this case, the orthogonal part of the maps is just a multiplication by $1$ or $-1$ and therefore, we include it in the contraction. The quantity
\begin{equation*}
  \Delta_n = \min\{|\fii_\iii(0)-\fii_\jjj(0)| : \iii,\jjj \in I^n \text{ such that } \iii \ne \jjj \text{ and } \lambda_\iii = \lambda_\jjj\}
\end{equation*}
is zero for arbitrary large $n$ if and only if there is an exact overlap. It is also easy to see that $\Delta_n \to 0$ at least exponentially for every $\Phi$. We say that $\Phi$ is \emph{exponentially separated} if there is $c>0$ such that $\Delta_n \ge c^n$ for arbitrary large $n$. Hochman \cite[Corollary 1.2]{Hochman2014} has shown that if $\dimh(X) < \min\{1,\dims(\Phi)\}$, then there is \emph{super-exponential condensation}, which means that $\Delta_n \to 0$ super-exponentially, $\lim_{n \to \infty} \tfrac{1}{n} \log\Delta_n = -\infty$. In other words, if $\Phi$ is exponentially separated, then there is no dimension drop. In particular, if $\Phi$ is defined by using algebraic parameters and there are no exact overlaps, then $\Phi$ is exponentially separated; see Hochman \cite[proof of Theorem 1.5]{Hochman2014}. Therefore, the dimension drop conjecture holds for all $\Phi$ defined by using algebraic parameters. Recently, Rapaport~\cite{Rap} strengthened this result to hold with the assumption that only the contraction ratios are algebraic.

In our main result, we show that Hochman's theorem, as stated, is still too weak to address the full conjecture. Let us define parametrized maps $\fii_i^{\lambda,t} \colon \R \to \R$ for $i \in \{1,2,3\}$ by
\begin{equation*}
  \fii_1^{\lambda,t}(x) = \lambda x, \qquad \fii_2^{\lambda,t}(x) = \lambda x + t, \qquad \fii_3^{\lambda,t}(x) = \lambda x + 1
\end{equation*}
for all $0 < \lambda < \tfrac13$ and $0 < t < \lambda/(1-\lambda)$. Write $\Phi_{\lambda,t} = (\fii_1^{\lambda,t},\fii_2^{\lambda,t},\fii_3^{\lambda,t})$ and let $X_{\lambda,t}$ be the associated self-similar set. Note that the restriction $t < \lambda/(1-\lambda) < (1-2\lambda)/(1-\lambda)$ guarantees that $\fii_1^{\lambda,t}(\conv(X_{\lambda,t})) \cap \fii_2^{\lambda,t}(\conv(X_{\lambda,t})) \ne \emptyset$ and $\fii_2^{\lambda,t}(\conv(X_{\lambda,t})) \cap \fii_3^{\lambda,t}(\conv(X_{\lambda,t})) = \emptyset$, where $\conv(A)$ is the convex hull of a given set $A$. Let us define three planar sets specific to this setting. The \emph{exact overlapping set} is
\begin{equation*}
  \EE = \{(\lambda,t) : \fii_\iii^{\lambda,t} = \fii_\jjj^{\lambda,t} \text{ for some finite sequences } \iii \ne \jjj \},
\end{equation*}
the \emph{dimension drop set} is
\begin{equation*}
  \DD = \{(\lambda,t) : \dimh(X_{\lambda,t}) < \dims(\Phi_{\lambda,t}) = -\log 3/\log \lambda \},
\end{equation*}
and the \emph{super-exponential condensation set} is
\begin{equation*}
 \CC = \{(\lambda,t) : \lim_{n \to \infty} \tfrac{1}{n}\log\Delta_n^{\lambda,t} = -\infty\},
\end{equation*}
where $\Delta_n^{\lambda,t} = \min\{|\fii_\iii^{\lambda,t}(0)-\fii_\jjj^{\lambda,t}(0)| : \iii,\jjj \in \{1,2,3\}^n \text{ such that } \iii \ne \jjj\}$. We are also interested in specifying the convergence speed in the super-exponential condensation set. Therefore, we define the \emph{$\eta$-condensation set} to be
\begin{equation*}
  \mathcal{C}_\eta = \{(\lambda,t) : \Delta_n^{\lambda,t} < \eta_n\text{ for all } n \in \N\} \subset \CC,
\end{equation*}
where $\eta=(\eta_n)_{n \in \N}$ is a given monotone decreasing sequence of positive real numbers such that $\lim_{n\to\infty}\frac{1}{n}\log\eta_n=-\infty$.

As discussed above, we trivially have $\EE \subset \DD$ and, by Hochman~\cite[Theorem~1.2]{Hochman2014}, $\DD \subset \CC$. Furthermore, by Hochman \cite[Theorem~1.10]{Hochmanrd}, we have $\dimh(\EE) = \dimp(\CC) = 1$, where $\dimp$ is the packing dimension. For the parametrized tuple $\Phi_{\lambda,t}$, the dimension drop conjecture is equivalent to $\DD \setminus \EE = \emptyset$. Very recently, by developing new techniques, Rapaport and Varj\'u~\cite[Corollary~1.4]{RapVar} showed that if $\DD\setminus\EE \ne \emptyset$, then $\dimh(\DD\setminus\EE) = 0$.

The following result shows that there exist self-similar sets having super-exponential condensation without exact overlaps.

\begin{theorem} \label{thm:main2}
If $\eta = (\eta_n)_{n \in \N}$ is a monotone decreasing sequence of positive real numbers such that $\lim_{n\to\infty}\frac{1}{n}\log\eta_n=-\infty$, then for the parametrized tuple $\Phi_{\lambda,t}$ defined above, the set $\CC_\eta \setminus \EE$ is uncountable.
\end{theorem}

Independent of our work, Baker \cite{Baker2019} has recently introduced a similar result; see also \cite{Baker2020}. In \cite{Baker2019}, he showed the existence of a self-similar set on the line, with at least 6 defining maps having rational contracting ratios, such that there is super-exponential condensation with arbitrary fast convergence speed but no exact overlaps. Chen~\cite{chen}, by adapting the method of Baker, obtained further examples of this type by using only 4 maps. Baker used continued fraction expansion and the Diophantine approximation of the translation parameters. This allowed Rapaport~\cite{Rap} to show that there exists a super-exponentially condensated self-similar set on the line having no dimension drop. Our result uses 3 maps and the proof relies on non-linear projections and the transversality method. We note that it is still unknown whether it is possible to achieve super-exponential condensation without exact overlaps by using only 2 maps. For instance, in the case of Bernoulli convolutions which is probably the simplest example of self-similar measures constructed by using 2 maps, this problem is strongly related to determining the distance of roots of polynomials with integer coefficients; see \cite[Question~3.8]{Hochmanproc}. In fact, unlike in the case of 3 maps, it is not possible to achieve arbitrary convergence speed in the super-exponential condensation with 2 maps having the same contraction ratio. Applying the bound of Mahler \cite{Mah}, one can show that if there are no exact overlaps, then $\Delta_n\geq n^{-cn}$ for some $c>0$; see \cite[Section~3.4]{Hochmanproc}. 

Finally, we characterize the dimension drop of the natural measure on homogeneous self-similar sets $X \subset \R$ by means of the \emph{average exponential condensation} defined by
\begin{equation*}
\Lambda(\gamma) = \liminf_{n\to\infty} \frac{1}{n} \sum_{\iii \in I^n} \frac{1}{\# I^n} \log\#\{\jjj\in I^n : |\fii_\iii(0)-\fii_\jjj(0)|\le\gamma^n\}
\end{equation*}
for all $\gamma > 0$. The \emph{natural measure} is the Borel probability measure $\mu$ on $X$ satisfying
\begin{equation*}
\mu = \frac{1}{\# I} \sum_{i \in I} \mu \circ \fii_i^{-1}.
\end{equation*}
Recall that the \emph{(lower) Hausdorff dimension} of $\mu$ is
\begin{equation*}
\dimh(\mu) = \inf\{\dimh(A):A \text{ is a Borel set such that }\mu(A)>0\}.
\end{equation*}
If $\lambda$ is the common contraction ratio of the maps $\fii_i$, then $\dimh(\mu) \le \dimh(X) \le \dims(\Phi) = -\log \#I/\log |\lambda|$ regardless of the translations.

\begin{proposition} \label{thm:main1}
	If $\Phi = (\fii_i)_{i \in I}$ is a homogeneous tuple of contractive similitudes acting on the real line such that $\lambda$ with $0<|\lambda|<1/\# I$ is the common contraction ratio of the maps $\fii_i$, $X \subset \R$ is the associated self-similar set, and $\mu$ is the natural measure on $X$, then
	\begin{equation*}
	\dimh(\mu) = \dims(\Phi) - \frac{\Lambda(\gamma)}{\log |\lambda|^{-1}}
	\end{equation*}
	for all $0<\gamma\le|\lambda|$. Furthermore, the limit inferior in the definition of $\Lambda(\gamma)$ is a limit and the value of $\Lambda(\gamma)$ does not depend on the choice of $0<\gamma\le|\lambda|$.
\end{proposition}

Finally, Theorem \ref{thm:main2} and Proposition \ref{thm:main1} introduce a possible way to disprove the dimension drop conjecture: If there exist $(\lambda,t) \in \CC \setminus \EE$ and $0 < \gamma \le \lambda$ such that
\begin{equation*}
  \Lambda^{\lambda,t}(\gamma) = \liminf_{n\to\infty} \frac{1}{n} \sum_{\iii \in I^n} \frac{1}{\# I^n} \log\#\{\jjj\in I^n : |\fii_\iii^{\lambda,t}(0)-\fii_\jjj^{\lambda,t}(0)|\le\gamma^n\} > 0,
\end{equation*}
then the dimension of the natural measure drops even though there are no exact overlaps.


\section{Super-exponential condensation}

In this section, we prove Theorem \ref{thm:main2}. Let us first observe that, for the parametrized tuple $\Phi_{\lambda,t}$, the canonical projection $\pi_{\lambda,t} \colon \{1,2,3\}^\N \to X_{\lambda,t}$ satisfies
\begin{equation*}
  \pi_{\lambda,t}(\iii)=\sum_{k=1}^\infty(\delta_{i_k}^3+t\delta_{i_k}^2)\lambda^{k-1}
\end{equation*}
for all $\iii = i_1i_2\cdots \in\{1,2,3\}^\N$, where
\begin{equation*}
\delta_i^j =
  \begin{cases}
    1, & \mbox{if } i=j, \\
    0, & \mbox{if } i \ne j.
  \end{cases}
\end{equation*}
Note that $\varphi_{\iii}^{\lambda,t}(0)=\pi_{\lambda,t}(\iii1^\infty)=\sum_{k=1}^n(\delta_{i_k}^3+t\delta_{i_k}^2)\lambda^{k-1}$ for all $\iii \in \{1,2,3\}^n$ and $n \in \N$, where $1^\infty$ is the infinite sequence containing only $1$'s.

\begin{lemma}\label{lem:changesys}
  Let $n \in \N$, $\iii=i_1\cdots i_n,\jjj=j_1\cdots j_n \in \{1,2,3\}^n$ be such that $i_1 \ne j_1$, and $0<\eps<\tfrac12$. Then
  \begin{equation*}
    |\varphi^{\lambda,t}_{\iii}(0)-\varphi^{\lambda,t}_{\jjj}(0)|<\varepsilon \quad\implies\quad
    \Biggl|t-\frac{\sum_{k=1}^{n}(\delta_{j_k}^3-\delta_{i_k}^3)\lambda^{k-1}}{\sum_{k=1}^{n}(\delta_{i_k}^2-\delta_{j_k}^2)\lambda^{k-1}}\Biggr|<2\varepsilon
  \end{equation*}
  and
  \begin{equation*}
    \Biggl|t-\frac{\sum_{k=1}^{n}(\delta_{j_k}^3-\delta_{i_k}^3)\lambda^{k-1}}{\sum_{k=1}^{n}(\delta_{i_k}^2-\delta_{j_k}^2)\lambda^{k-1}}\Biggr|<\varepsilon \quad\implies\quad
    |\varphi^{\lambda,t}_{\iii}(0)-\varphi^{\lambda,t}_{\jjj}(0)|<2\varepsilon.
  \end{equation*}
\end{lemma}

\begin{proof}
  Since
  $$
    \varphi^{\lambda,t}_{\iii}(0)-\varphi^{\lambda,t}_{\jjj}(0)=\sum_{k=1}^{n}(\delta_{i_k}^3-\delta_{j_k}^3)\lambda^{k-1}+t\sum_{k=1}^{n}(\delta_{i_k}^2-\delta_{j_k}^2)\lambda^{k-1},
  $$
  we see that both claims follow if we can show that
  \begin{equation*}
    \tfrac12 \le \Biggl| \sum_{k=1}^n (\delta_{i_k}^2-\delta_{j_k}^2)\lambda^{k-1} \Biggr| \le 2.
  \end{equation*}
  The lower bound is needed in the first claim and the upper bound in the second claim. To show the lower bound, we may thus assume that $|\varphi^{\lambda,t}_{\iii}(0)-\varphi^{\lambda,t}_{\jjj}(0)|<\tfrac12$. Since $0 < \lambda < \tfrac13$, this is possible only if $|\delta_{i_1}^2-\delta_{j_1}^2|=1$. Indeed, the distance of $\varphi_1([0,(1-\lambda)^{-1}])$ and $\varphi_3([0,(1-\lambda)^{-1}])$ is $(1-2\lambda)/(1-\lambda)>\frac12$ and so at least one of the first symbols must be a 2. Therefore, $|\sum_{k=1}^n (\delta_{i_k}^2-\delta_{j_k}^2)\lambda^{k-1}| \ge 1-\sum_{k=1}^\infty \lambda^k = (1-2\lambda)/(1-\lambda) > \tfrac12$ as claimed. The upper bound is trivial since $|\sum_{k=1}^n (\delta_{i_k}^2-\delta_{j_k}^2)\lambda^{k-1}| \le \sum_{k=0}^\infty \lambda^k = (1-\lambda)^{-1} < \tfrac32$.
\end{proof}

Lemma~\ref{lem:changesys} tells us that in order to achieve super-exponential condensation, the parameter $t$ must be contained in a super-exponential neighbourhood of a ratio of the form
\begin{equation*}
  \frac{\sum_{k=1}^{n}(\delta_{j_k}^3-\delta_{i_k}^3)\lambda^{k-1}}{\sum_{k=1}^{n}(\delta_{i_k}^2-\delta_{j_k}^2)\lambda^{k-1}}.
\end{equation*}
We shall show that such ratios are certain non-linear projections of an induced self-similar set in the plane.

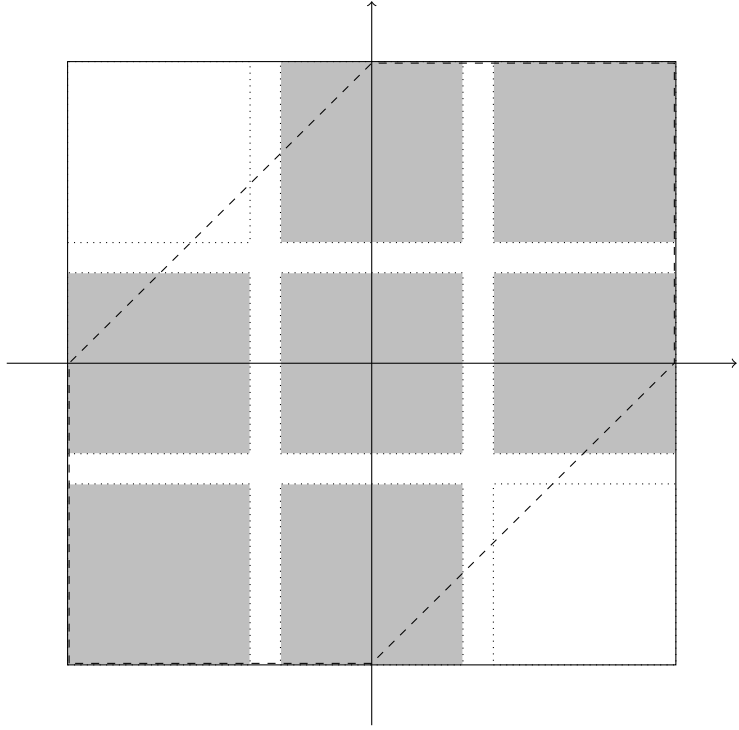
\begin{figure}[t]
	\begin{tikzpicture}[scale=0.4]
	\draw [dotted] (0,14) rectangle (6,20);
	
	\path[draw=white,fill=lightgray] (7,14) rectangle (13,20);
	\draw [dotted] (7,14) rectangle (13,20);
	
	\path[draw=white,fill=lightgray] (14,14) rectangle (20,20);
	\draw [dotted] (14,14) rectangle (20,20);
	
	\path[draw=white,fill=lightgray] (0,7) rectangle (6,13);
	\draw [dotted] (0,7) rectangle (6,13);
	
	\path[draw=white,fill=lightgray] (7,7) rectangle (13,13);
	\draw [dotted] (7,7) rectangle (13,13);
	
	\path[draw=white,fill=lightgray] (14,7) rectangle (20,13);
	\draw [dotted] (14,7) rectangle (20,13);
	
	\path[draw=white,fill=lightgray] (0,0) rectangle (6,6);
	\draw [dotted] (0,0) rectangle (6,6);
	
	\path[draw=white,fill=lightgray] (7,0) rectangle (13,6);
	\draw [dotted] (7,0) rectangle (13,6);
	
	\draw [dotted] (14,0) rectangle (20,6);
	
	\draw (0,0) -- (20,0) -- (20,20) -- (0,20) -- cycle;
	
	\draw [dashed] (0.05,0.05) -- (0.05,10) -- (10,19.95) -- (19.95,19.95) -- (19.95,10) -- (10,0.05) -- cycle;
	
	\draw [->] (-2,10) -- (22,10) node [] {};
	\draw [->] (10,-2) -- (10,22) node [] {};
	\end{tikzpicture}
	\caption{Illustration for the self-similar set $K_\lambda$ and the convex hull $P_\lambda$.}
	\label{fig:illustration}
\end{figure}

Let $J = \{(0,0),(-1,0),(-1,-1),(1,0),(1,1),(0,-1),(0,1)\}$ and define $S^\lambda_{(i,j)} \colon \R^2 \to \R^2$ by setting
\begin{equation*}
  S^\lambda_{(i,j)}(x,y) = (\lambda x+i, \lambda y+j)
\end{equation*}
for all $(x,y) \in \R^2$ and $(i,j) \in J$. Write $\Psi_\lambda = (S^\lambda_{(i,j)})_{(i,j) \in J}$ and let $K_\lambda \subset \R^2$ be the associated self-similar set; see Figure~\ref{fig:illustration} for an illustration. The map $\beta \colon \{1,2,3\} \times \{1,2,3\} \to J$ defined by setting $\beta(i,j) = (\delta_{j}^3-\delta_{i}^3,\delta_{i}^2-\delta_{j}^2)$ is clearly one-to-one outside the diagonal. We extend the map $\beta$ to $\{1,2,3\}^n \times \{1,2,3\}^n \to J^n$ for all $n \in \N$ and to $\{1,2,3\}^\N \times \{1,2,3\}^\N \to J^\N$ in a natural way: for example, if $\iii=i_1\cdots i_n$ and $\jjj=j_1\cdots j_n$ for some $n \in \N$, then $\beta(\iii,\jjj)$ is defined to be $\beta(i_1,j_1)\cdots\beta(i_n,j_n)$. Finally, we define a non-linear projection $\proj \colon \{(x,y)\in\R^2:y\ne 0\} \to \R$ by setting
\begin{equation*}
  \proj(x,y) = \frac{x}{y}
\end{equation*}
for all $x \in \R$ and $y \in \R \setminus \{0\}$. The following lemma basically restates Lemma \ref{lem:changesys} in terms of the projection. If $\iii=i_1\cdots i_n$ and $\jjj$ are finite sequences, then we write $\iii\land\jjj$ for their common beginning and $\sigma(\iii)=i_2\cdots i_n$.

\begin{lemma}\label{lem:changesys2}
  Let $n \in \N$, $\iii,\jjj \in \{1,2,3\}^n$ be such that $\iii \ne \jjj$, and $0<\eps<\tfrac12\lambda^{|\iii\land\jjj|}$. Then
  \begin{equation*}
    |\varphi^{\lambda,t}_{\iii}(0)-\varphi^{\lambda,t}_{\jjj}(0)|<\varepsilon \quad\implies\quad
    |t-\proj(S^\lambda_{\beta(\iii,\jjj)}(0,0))|<2\lambda^{-|\iii\land\jjj|}\varepsilon
  \end{equation*}
  and
  \begin{equation*}
    |t-\proj(S^\lambda_{\beta(\iii,\jjj)}(0,0))|<\lambda^{-|\iii\land\jjj|}\varepsilon \quad\implies\quad
    |\varphi^{\lambda,t}_{\iii}(0)-\varphi^{\lambda,t}_{\jjj}(0)|<2\varepsilon.
  \end{equation*}
\end{lemma}

\begin{proof}
  Note that $\proj(S^\lambda_{\beta(\iii,\jjj)}(0,0))=\proj(S^\lambda_{\beta(\sigma^{|\iii\wedge\jjj|}(\iii),\sigma^{|\iii\wedge\jjj|}(\jjj))}(0,0))$ whenever $\iii\neq\jjj$. Therefore, as $|\varphi^{\lambda,t}_\iii(0)-\varphi^{\lambda,t}_\jjj(0)|=\lambda^{|\iii\wedge\jjj|}|\varphi^{\lambda,t}_{\sigma^{\iii\wedge\jjj}(\iii)}(0)-\varphi^{\lambda,t}_{\sigma^{\iii\wedge\jjj}(\jjj)}(0)|$, the proof follows from Lemma~\ref{lem:changesys}.
\end{proof}

We have thus transformed the problem to a study of non-linear projections. Note that, for every finite sequence $\kkk$ containing a symbol with non-zero second coordinate, we have
\begin{equation*}
  \proj(S_\kkk^\lambda(0,0)) = \lambda^m\frac{\pm1+\sum_{\ell=1}^{k}\alpha_{\ell,1}\lambda^k}{\pm1+\sum_{\ell=1}^{n}\alpha_{\ell,2}\lambda^k},
\end{equation*}
where $m$ is the difference of the positions of the appearance of the first non-zero first coordinate and first non-zero second coordinate and $\alpha_{i,j}\in\{-1,0,1\}$. Hence,
\begin{align}
  \lambda_0^m(1-2\lambda_0) \leq |\proj(S_\kkk^{\lambda_0}(0,0))| &\leq \frac{\lambda_0^m}{1-2\lambda_0}, \label{eq:genbound1} \\
  \bigl|\tfrac{\mathrm{d}}{\mathrm{d}\lambda}\proj(S_\kkk^\lambda(0,0))\big|_{\lambda=\lambda_0}\bigr| &\leq \frac{2|m|\lambda_0^{m-1}}{(1-\lambda_0)(1-2\lambda_0)} \label{eq:genbound2}
\end{align}
for all $\tfrac14<\lambda<\tfrac13$.

Our first concrete goal now is to find finite words $\kkk$ for which $\proj(S_\kkk^{\lambda}(K_\lambda))$ is an interval for a range of $\lambda$'s. After some preliminary lemmas we will achieve this in Lemma \ref{prop:projint}. Let
$$
  P_\lambda=\mathrm{conv}\biggl(\biggl\{\biggl(\frac{i}{1-\lambda},\frac{j}{1-\lambda}\biggr)\biggr\}_{(i,j)\in J}\biggr).
$$
Since each extremal point of $P_\lambda$ is the fixed point of some map $S_{(i,j)}^\lambda$ and all the maps $S_{(i,j)}^\lambda$ are direction preserving homotethies with $S_{(i,j)}^\lambda(P_\lambda) \subset P_\lambda$, we see that $K_\lambda \subset P_\lambda$. See again Figure~\ref{fig:illustration} for an illustration.

\begin{lemma}\label{lem:projhull}
  If $(a,b)\in S^\lambda_{(0,1)}(P_\lambda)\cap(0,\infty)\times\R$ and $(0,0)^n = (0,0)\cdots(0,0) \in J^n$ for all $n \in \N$, then
  $$
    \proj(S^\lambda_{(0,0)^n}(P_\lambda)+(a,b))=\Biggl[\frac{a-\frac{\lambda^n}{1-\lambda}}{b},\frac{a+\frac{\lambda^n}{1-\lambda}}{b}\Biggr]
  $$
  for all $n \in \N$ with $a\geq \frac{\lambda^n}{1-\lambda}$.
\end{lemma}

\begin{proof}
  Observe first that $S^\lambda_{(0,1)}(P_\lambda)$ lies strictly above the $x$-axis: if $(a,b)\in S^\lambda_{(0,1)}(P_\lambda)$, then $b>\frac{1-2\lambda}{1-\lambda}>\frac12$. Since the line segment connecting $(a-\frac{\lambda^n}{1-\lambda},b)$ and $(a+\frac{\lambda^n}{1-\lambda},b)$ is contained in $S^\lambda_{(0,0)^n}(P_\lambda)+(a,b)$, the projection contains the claimed interval. Let us show that
  \begin{equation*}
    \frac{a-\frac{\lambda^n}{1-\lambda}}{b}\leq\frac{x}{y}
  \end{equation*}
  for all $(x,y)\in S^\lambda_{(0,0)^n}(P_\lambda)+(a,b)$. The assumption $a\geq\frac{\lambda^n}{1-\lambda}$ guarantees that $x\geq0$ for all $(x,y)\in S^\lambda_{(0,0)^n}(P_\lambda)+(a,b)$. Thus, by decreasing $x$ and increasing $y$, we see that there exists $s\in[0,1]$ such that
   \begin{equation*}
  \frac{a-s\frac{\lambda^n}{1-\lambda}}{b+(1-s)\frac{\lambda^n}{1-\lambda}}\leq\frac{x}{y}.
  \end{equation*}
  By differentiating, the function
  \begin{equation*}
    s \mapsto \frac{a-s\frac{\lambda^n}{1-\lambda}}{b+(1-s)\frac{\lambda^n}{1-\lambda}}
  \end{equation*}
  can easily be seen to be monotone on $[0,1]$. It is therefore enough to check that the endpoints satisfy
  \begin{equation*}
    \frac{a-\frac{\lambda^n}{1-\lambda}}{b}\leq\frac{a}{b+\frac{\lambda^n}{1-\lambda}}.
  \end{equation*}
  Since $0<\lambda<\tfrac13$, a simple calculation shows that this holds for every $n\in\N$.

  Let us then show that
  \begin{equation*}
    \frac{x}{y} \le \frac{a+\frac{\lambda^n}{1-\lambda}}{b}
  \end{equation*}
  for all $(x,y)\in S^\lambda_{(0,0)^n}(P_\lambda)+(a,b)$. Similarly, by increasing $x$ and decreasing $y$, we see that there exists $s\in[0,1]$ such that
  \begin{equation*}
    \frac{x}{y}\leq \frac{a+s\frac{\lambda^n}{1-\lambda}}{b-(1-s)\frac{\lambda^n}{1-\lambda}}.
  \end{equation*}
  Since the right-hand side is monotone on $[0,1]$, it suffices to check that
  \begin{equation*}
    \frac{a}{b-\frac{\lambda^n}{1-\lambda}}\leq\frac{a+\frac{\lambda^n}{1-\lambda}}{b}.
  \end{equation*}
  A simple calculation shows that this holds if and only if $\frac{\lambda^n}{1-\lambda}\le b-a$. Clearly, $b\geq a+\frac{1-2\lambda}{1-\lambda}$ for $(a,b)\in S_{(0,1)}(P_\lambda)\cap[0,\infty)\times\R$, and since $0<\lambda<\tfrac13$, we have
  \begin{equation*}
    \frac{\lambda^n}{1-\lambda}\le\frac{\lambda}{1-\lambda}<\frac{1}{2}<\frac{1-2\lambda}{1-\lambda}\le b-a.
  \end{equation*}
  The proof is finished.
\end{proof}

\begin{lemma}\label{lem:indproj}
  If $(a,b)\in S^\lambda_{(0,1)}(P_\lambda)\cap(0,\infty)\times\R$ and $N \in \N$ is such that
  \begin{equation*}
    \lambda^N\leq\min\biggl\{2a-\frac{1-3\lambda}{1-\lambda}b,\frac{2\lambda}{1-2\lambda}b-\frac{1-\lambda}{1-2\lambda}a\biggr\},
  \end{equation*}
  then
  \begin{equation*}
    \proj((a,b)+S^\lambda_{(0,0)^n}(P_\lambda))=\bigcup_{(i,j)\in J}\proj((a+\lambda^ni,b+\lambda^nj)+S^\lambda_{(0,0)^{n+1}}(P_\lambda))
  \end{equation*}
  for all $n \ge N$ with $a\geq\frac{\lambda^n}{1-\lambda}$.
\end{lemma}

\begin{proof}
	Fix $(a,b)\in S^\lambda_{(0,1)}(P_\lambda)\cap(0,\infty)\times\R$ and let $n\geq N$ be such that $a\geq\frac{\lambda^n}{1-\lambda}$. Hence, $a+i\lambda^n\geq\frac{\lambda^n}{1-\lambda}-\lambda^n=\frac{\lambda^{n+1}}{1-\lambda}$. By Lemma~\ref{lem:projhull}, it is thus enough to show that
  $$
    \left[\frac{a-\frac{\lambda^n}{1-\lambda}}{b},\frac{a+\frac{\lambda^n}{1-\lambda}}{b}\right]=\bigcup_{(i,j)\in J}\left[\frac{a+i\lambda^n-\frac{\lambda^{n+1}}{1-\lambda}}{b+j\lambda^n},\frac{a+i\lambda^n+\frac{\lambda^{n+1}}{1-\lambda}}{b+j\lambda^n}\right].
  $$
  In fact, it suffices to show that the consecutive intervals in the above union have non-empty intersection. The order of the intervals corresponds to the following order in $J$:
  $$
    (-1,0),(-1,-1),(0,1),(0,0),(0,-1),(1,1),(1,0).
  $$
  Hence, we have to check that the following inequalities hold:
  \begin{align}
    \frac{a-\lambda^n-\frac{\lambda^{n+1}}{1-\lambda}}{b-\lambda^n}&\leq\frac{a-\lambda^n+\frac{\lambda^{n+1}}{1-\lambda}}{b},\label{eq:l1}\\
    \frac{a-\frac{\lambda^{n+1}}{1-\lambda}}{b+\lambda^n}&\leq\frac{a-\lambda^n+\frac{\lambda^{n+1}}{1-\lambda}}{b-\lambda^n},\label{eq:l2}\\
    \frac{a-\frac{\lambda^{n+1}}{1-\lambda}}{b}&\leq\frac{a+\frac{\lambda^{n+1}}{1-\lambda}}{b+\lambda^n},\label{eq:l3}\\
    \frac{a-\frac{\lambda^{n+1}}{1-\lambda}}{b-\lambda^n}&\leq\frac{a+\frac{\lambda^{n+1}}{1-\lambda}}{b},\label{eq:l4}\\
    \frac{a+\lambda^n-\frac{\lambda^{n+1}}{1-\lambda}}{b+\lambda^n}&\leq\frac{a+\frac{\lambda^{n+1}}{1-\lambda}}{b-\lambda^n},\label{eq:l5}\\
    \frac{a+\lambda^n-\frac{\lambda^{n+1}}{1-\lambda}}{b}&\leq\frac{a+\lambda^n+\frac{\lambda^{n+1}}{1-\lambda}}{b+\lambda^n}.\label{eq:l6}
  \end{align}
  The inequality \eqref{eq:l1} holds if and only if $a\leq \frac{2\lambda}{1-\lambda}b+\lambda^n\frac{1-2\lambda}{1-\lambda}$. Recalling that $a\leq\frac{\lambda}{1-\lambda}$ and $b\geq\frac{1-2\lambda}{1-\lambda}$, we see that this holds if $\frac{\lambda}{1-\lambda}\leq \frac{2\lambda(1-2\lambda)}{(1-\lambda)^2}$, which is true since $0<\lambda<\tfrac13$.
  The inequalities \eqref{eq:l2} and \eqref{eq:l5} hold if and only if $2a-\lambda^n\geq\frac{1-3\lambda}{1-\lambda}b$, which is true by the assumption. The inequality \eqref{eq:l3} holds if and only if $a\leq \frac{2\lambda}{1-\lambda}b+\frac{\lambda^{n+1}}{1-\lambda}$, which can be seen to hold similarly as with the inequality \eqref{eq:l1}. Finally, the inequality \eqref{eq:l4} holds if and only if $a\leq \frac{2\lambda}{1-\lambda}b-\frac{\lambda^{n+1}}{1-\lambda}$ and \eqref{eq:l6} if and only if $a\leq \frac{2\lambda}{1-\lambda}b-\frac{\lambda^{n}(1-2\lambda)}{1-\lambda}$, which are true again by the assumption.
\end{proof}

We are now ready to show that the projection of $K_\lambda$ contains an interval for a range of $\lambda$'s.

\begin{lemma}\label{prop:projint}
  If $\tfrac14<\lambda<\tfrac13$ and $\kkk \in \bigcup_{n=3}^\infty J^n$ satisfies $\kkk|_3=(0,1)(1,1)(0,-1)$ or $\kkk|_3=(0,1)(1,0)(0,1)$, then $\proj(S_\kkk^{\lambda}(K_\lambda))$ is an interval with center $\proj(S^\lambda_{\kkk}(0,0))$ and with length at least $2\lambda^{|\kkk|}$ and at most $3\lambda^{|\kkk|}$.
\end{lemma}

\begin{proof} 
	The claim follows if it holds that $\proj(S^\lambda_\kkk(K_\lambda))=\proj(S^\lambda_\kkk(P_\lambda))$. For that, it is enough to show that
  \begin{equation}\label{eq:enough}
    \proj(S^\lambda_{\kkk\jjj}(P_\lambda))=\bigcup_{(i,j)\in J}\proj(S^\lambda_{\kkk\jjj(i,j)}(P_\lambda)).
  \end{equation}
  for every finite sequence $\jjj$. Indeed, if \eqref{eq:enough} holds, then for every $n\in\N$ we have
  \begin{equation*}
    \bigcup_{\jjj\in J^n}\proj(S^\lambda_{\kkk\jjj}(P_\lambda))=\bigcup_{\jjj\in J^{n-1}}\bigcup_{(i,j)\in J}\proj(S^\lambda_{\kkk\jjj(i,j)}(P_\lambda))=\bigcup_{\jjj\in J^{n-1}}\proj(S^\lambda_{\kkk\jjj}(P_\lambda))
  \end{equation*}
  and therefore, by the definition of the self-similar set $K_\lambda$,
	\begin{equation*}
    \proj(S^\lambda_{\kkk}(K_\lambda)) = \proj\biggl(S^\lambda_{\kkk}\biggl(\bigcap_{n=0}^\infty\bigcup_{\jjj\in J^n}S_\jjj^\lambda(P_\lambda)\biggr)\biggr) = \bigcap_{n=0}^\infty\bigcup_{\jjj\in J^n}\proj(S^\lambda_{\kkk\jjj}(P_\lambda))=\proj(S^\lambda_{\kkk}(P_\lambda)).
    \end{equation*}

  To verify \eqref{eq:enough}, it is enough to check whether the assumptions of Lemma~\ref{lem:indproj} hold. Let $(a,b)$ be the middle point of $S^\lambda_{\kkk\jjj}(P_\lambda)$. It is easy to see that then $\lambda-\frac{\lambda^3}{1-\lambda}\leq a\leq \lambda+\frac{\lambda^3}{1-\lambda}$ and $1+\lambda^2-\frac{\lambda^3}{1-\lambda}\leq b\leq 1+\lambda-\lambda^2+\frac{\lambda^3}{1-\lambda}$.
  Hence, the inequality
  \begin{align*}
    \lambda^3\leq\min\biggl\{2\biggl(\lambda-\frac{\lambda^3}{1-\lambda}\biggr)-&\frac{1-3\lambda}{1-\lambda}\biggl(1+\lambda-\lambda^2+\frac{\lambda^3}{1-\lambda}\biggr),\\
    &{\frac{2\lambda}{1-2\lambda}}\biggl(1+\lambda^2-\frac{\lambda^3}{1-\lambda}\biggr)-\frac{1-\lambda}{{1-2\lambda}}\biggl(\lambda+\frac{\lambda^3}{1-\lambda}\biggr)\biggr\}
  \end{align*}
  clearly implies the assumptions of Lemma~\ref{lem:indproj}. Numerical calculations show that the above inequality is valid for all $\tfrac14<\lambda<\tfrac13$. Finally, by Lemma~\ref{lem:projhull}, the center of $\proj(S^\lambda_\kkk(P_\lambda))=\proj(S^\lambda_\kkk(K_\lambda))$ is $\proj(S^\lambda_\kkk(0,0))$ and its length is $\frac{2\lambda^{|\kkk|}}{b(1-\lambda)}$, where $b$ is the $y$ coordinate of $\proj(S^\lambda_\kkk(0,0))$. Since $1+\lambda^2-\frac{\lambda^3}{1-\lambda}\leq b\leq 1+\lambda-\lambda^2+\frac{\lambda^3}{1-\lambda}$, we have $\frac{2\lambda^{|\kkk|}}{1-2\lambda^2-2\lambda^3}\leq\frac{2\lambda^{|\kkk|}}{b(1-\lambda)}\leq\frac{2\lambda^{|\kkk|}}{1-\lambda+\lambda^2-2\lambda^3}$. One can show by numerical computations that $2\leq \frac{2}{1-2\lambda^2-2\lambda^3}$ and $\frac{2}{1-\lambda+\lambda^2-2\lambda^3}\leq 3$ for every $\lambda\in[\tfrac14,\tfrac13]$.
\end{proof}

We will next show that the projection is transversal in this region of $\lambda$'s.

\begin{lemma}\label{lem:trans}
  There exists $\delta>0$ such that for every $\tfrac14<\lambda_0<\tfrac13$ and $\kkk,\lll\in\bigcup_{n=5}^\infty J^n$ with $\kkk|_5=(0,1)(1,1)(0,-1)(0,-1)(1,0)$ and $\lll|_5=(0,1)(1,0)(0,1)(0,1)(-1,0)$ we have
  $$
    \delta<\tfrac{\mathrm{d}}{\mathrm{d}\lambda}(\proj(S_\kkk^{\lambda}(0,0))-\proj(S_\lll^\lambda(0,0)))\big|_{\lambda=\lambda_0}<\delta^{-1}.
  $$
\end{lemma}

\begin{proof}
  The proof relies on numerical calculations. By our assumption on $\kkk$ and $\lll$, we have
  $$
    \proj(S_\kkk^{\lambda}(0,0))-\proj(S_\lll^\lambda(0,0))=\frac{\lambda+\lambda^4+a(\lambda)}{1+\lambda-\lambda^2-\lambda^3+b(\lambda)}-\frac{\lambda-\lambda^4+c(\lambda)}{1+\lambda^2+\lambda^3+d(\lambda)},
  $$
  where the functions $a(\lambda)$, $b(\lambda)$, $c(\lambda)$, and $d(\lambda)$ have the form $\sum_{k=5}^\infty\delta_k\lambda^k$, where $\delta_k\in\{-1,0,1\}$. Therefore, we see that
  \begin{equation*}
    \tfrac{\mathrm{d}}{\mathrm{d}\lambda}(\proj(S_\kkk^{\lambda}(0,0))-\proj(S_\lll^\lambda(0,0))) = A(\lambda)-B(\lambda),
  \end{equation*}
  where
  \begin{align*}
    A(\lambda) &= \frac{(1+4\lambda^3+a'(\lambda))(1+\lambda-\lambda^2-\lambda^3+b(\lambda))-(\lambda+\lambda^4+a(\lambda))(1-2\lambda-3\lambda^2+b'(\lambda))}{(1+\lambda-\lambda^2-\lambda^3+b(\lambda))^2}, \\
    B(\lambda) &= \frac{(1-4\lambda^3+c'(\lambda))(1+\lambda^2+\lambda^3+d(\lambda))-(\lambda-\lambda^4+c(\lambda))(2\lambda+3\lambda^2+d'(\lambda))}{(1+\lambda^2+\lambda^3+d(\lambda))^2}.
  \end{align*}
  Since
  \begin{align}
  \max\{|a(\lambda)|, |b(\lambda)|, |c(\lambda)|, |d(\lambda)|\} &\le \frac{\lambda^5}{1-\lambda},\label{eq:thisagain} \\
  \max\{|a'(\lambda)|, |b'(\lambda)|, |c'(\lambda)|, |d'(\lambda)|\} &\le \frac{\lambda^4(5-4\lambda)}{(1-\lambda)^2},
  \end{align}
  we have the estimates
  \begin{equation*}
    A(\lambda)\geq\frac{\bigl(1+4\lambda^3-\frac{\lambda^4(5-4\lambda)}{(1-\lambda)^2}\bigr)\bigl(1+\lambda-\lambda^2-\lambda^3-\frac{\lambda^5}{1-\lambda}\bigr)-\bigl(\lambda+\lambda^4+\frac{\lambda^5}{1-\lambda}\bigr)\bigl(1-2\lambda-3\lambda^2+\frac{\lambda^4(5-4\lambda)}{(1-\lambda)^2}\bigr)}{(1+\lambda-\lambda^2-\lambda^3+b(\lambda))^2}
  \end{equation*}
  and
  \begin{equation*}
    B(\lambda)\leq\frac{\bigl(1-4\lambda^3+\frac{\lambda^4(5-4\lambda)}{(1-\lambda)^2}\bigr)\bigl(1+\lambda^2+\lambda^3+\frac{\lambda^5}{1-\lambda}\bigr)-\bigl(\lambda-\lambda^4-\frac{\lambda^5}{1-\lambda}\bigr)\bigl(2\lambda+3\lambda^2-\frac{\lambda^4(5-4\lambda)}{(1-\lambda)^2}\bigr)}{(1+\lambda^2+\lambda^3+d(\lambda))^2}.
  \end{equation*}
  Numerical calculations show that both numerators appearing in the estimates above are strictly larger than $0.8$ for all $\lambda\in\left(\frac14,\frac13\right)$. Thus, we may apply \eqref{eq:thisagain} for the denominator as well, and therefore, numerical calculations show that $A(\lambda)-B(\lambda) \ge 0.057$ for all $\lambda\in\left(\frac14,\frac13\right)$. Since the other inequality follows by \eqref{eq:genbound2} in a straightforward way, we have finished the proof.
\end{proof}

From now on, we fix a monotone decreasing sequence $\eta = (\eta_n)_{n \in \N}$ of positive real numbers such that $\lim_{n\to\infty}\frac{1}{n}\log\eta_n=-\infty$. It follows that there exists a constant $C>0$ such that
\begin{equation}\label{eq:exp}
  \sum_{k=n}^\infty\eta_k\leq C\eta_n
\end{equation}
for all $n\in\N$. Relying on transversality, we will construct a Cantor set of super-exponentially condensated tuples. Observe that without loss of generality, we may assume that $\delta>0$ in Lemma~\ref{lem:trans} is small, for instance $\delta<\tfrac12$.

\begin{lemma}\label{lem:const}
 Let $0<\delta<\tfrac12$ be as in Lemma~\ref{lem:trans}. Let $\varepsilon>0$, $\tfrac14+3\delta^{-1}\varepsilon<\lambda<\tfrac13-3\delta^{-1}\varepsilon$, and let $\kkk,\lll\in\bigcup_{n=5}^\infty J^n$ with $\kkk|_5=(0,1)(1,1)(0,-1)(0,-1)(1,0)$ and $\lll|_5=(0,1)(1,0)(0,1)(0,1)(-1,0)$ or vice versa. If $\eta_{|\lll|}<\delta^{-1}\varepsilon$ and
  $$
    |\proj(S_\kkk^\lambda(0,0))-\proj(S_\lll^\lambda(0,0))|<\varepsilon,
  $$
  then there exist disjoint closed intervals $I,I'\subset[\lambda-3\delta^{-1}\varepsilon,\lambda+3\delta^{-1}\varepsilon]$ of length $\eta_{|\lll|}$ such that
  \begin{equation}\label{eq:propneed}
    \tfrac{1}{2}\delta\eta_{|\lll|}<|\proj(S_\kkk^{\lambda^*}(0,0))-\proj(S_\lll^{\lambda^*}(0,0))|<\tfrac{3}{2}\delta^{-1}\eta_{|\lll|}
  \end{equation}
  for all $\lambda^*\in I\cup I'$.
\end{lemma}

\begin{proof}
 By Lemma~\ref{lem:trans}, the derivative of the map $\lambda \mapsto \proj(S_\kkk^\lambda(0,0))-\proj(S_\lll^\lambda(0,0))$ lies in $[\delta,\delta^{-1}]$ for $\lambda\in(1/4,1/3)$. Hence, there exists a unique $\lambda_1\in[\lambda-\delta^{-1}\varepsilon,\lambda+\delta^{-1}\varepsilon]$ such that
  $$
    \proj(S_\kkk^{\lambda_1}(0,0))-\proj(S_\lll^{\lambda_1}(0,0))=0.
  $$
  Choose $I=[\lambda_1-\tfrac32\eta_{|\lll|},\lambda_1-\tfrac12\eta_{|\lll|}]$ and $I'=[\lambda_1+\tfrac12\eta_{|\lll|},\lambda_1+\tfrac32\eta_{|\lll|}]$. By the mean value theorem, \eqref{eq:propneed} holds for every $\lambda^*\in I\cup I'$. Note that $\lambda-3\delta^{-1}\varepsilon \le \lambda-\delta^{-1}\varepsilon-2\eta_{|\lll|} \le \lambda_1-\tfrac32\eta_{|\lll|}$ and, similarly, $\lambda_1+\tfrac32\eta_{|\lll|} \le \lambda+3\delta^{-1}\eps$.
\end{proof}

Next we state our main technical lemma which will be used to construct the claimed uncountable set in $\CC_\eta\setminus\EE$. The idea is to model the construction by a binary tree. To that end, let $\Omega_*=\bigcup_{n=0}^\infty\{0,1\}^n$ and $\Omega=\{0,1\}^\N$.

\begin{lemma}\label{prop:construction}
	There exists an injective map $\kkk\colon\Omega_* \to \bigcup_{n=1}^\infty J^n$ and, for every $\omega\in\Omega_*$, there exists a closed interval $I_\omega \subset (\tfrac14,\tfrac13)$ such that the following five conditions hold:
	\begin{enumerate}
		\item\label{it:k} $\min_{\omega\in\{0,1\}^{n+1}}|\kkk(\omega)|>\max_{\tau\in\{0,1\}^n}|\kkk(\tau)|$ for all $n\in\N$. Moreover, for every $\omega\in\Omega_*$, $\kkk(\omega)$ begins with either $(0,1)(1,1)(0,-1)(0,-1)(1,0)$ or $(0,1)(1,0)(0,1)(0,1)(-1,0)$ .
		\item\label{it:int1} For every $\omega\in\Omega_*$, we have
    \begin{equation*}
      I_{\omega0}\cup I_{\omega1}\subset I_\omega\text{ and }I_{\omega0}\cap I_{\omega1}=\emptyset.
    \end{equation*}
		\item\label{it:int2} $\diam(I_\omega)=\eta_{|\kkk(\omega)|}$.
		\item\label{it:dist} For every $\omega\in\Omega_*$, $i\in\{0,1\}$, and $\lambda\in I_{\omega i}$, we have
    \begin{equation*}
      |\proj(S_{\kkk(\omega)}^{\lambda}(0,0))-\proj(S_{\kkk(\omega i)}^{\lambda}(0,0))|<\tfrac{3}{2}\delta^{-1}\eta_{|\kkk(\omega i)|},
    \end{equation*}
		and
    \begin{equation*}
      \min\bigl\{|\proj(S_{\kkk(\omega i)}^{\lambda}(0,0))-\proj(S_{\iii}^{\lambda}(0,0))|:|\iii|\leq\max_{\tau\in\{0,1\}^{|\omega|}}|\kkk(\tau)|\bigr\}\geq \tfrac{1}{2}\delta\eta_{|\kkk(\omega i)|},
    \end{equation*}
    where $\delta>0$ is as in Lemma~\ref{lem:trans}.
		\item\label{it:where} For every $\omega\in\Omega_*$ and $\tfrac14<\lambda<\tfrac13$,
    \begin{equation*}
      0\leq\proj(S_{\kkk(\omega)}^{\lambda}(0,0))\leq\frac{\lambda}{1-\lambda}.
    \end{equation*}
	\end{enumerate}
\end{lemma}

\begin{proof}
	Without loss of generality, we may assume that $\delta<1/144$. Let $\kkk=(0,1)(1,1)(0,-1)(0,-1)(1,0)$ and $\lll=(0,1)(1,0)(0,1)(0,1)(-1,0)$. Algebraic calculations show that the unique solution of the equation
  \begin{equation*}
    \proj(S_\kkk^{\lambda}(0,0))=\proj(S_\lll^{\lambda}(0,0))
  \end{equation*}
	in the interval $(\tfrac14,\tfrac13)$ is $\lambda'=(\sqrt{13}-3)/2\approx0.302$. Choose $K\in\N$ to be the smallest natural number such that $[\lambda'- \tfrac12  \eta_K,\lambda'+\tfrac12 \eta_K]\subset(\tfrac14,\tfrac13)$ and $2(\tfrac14)^k>\tfrac32\delta^{-1}\eta_k$ for every $k\geq K$. Let us define
  \begin{align*}
    \kkk(\emptyset) &=
    \begin{cases}
      \kkk, &\text{if } |\kkk|\geq K, \\
      \kkk(0,0)^{K-|\kkk|}, &\text{if }|\kkk|< K,
    \end{cases} \\ 
    I_\emptyset &= [\lambda'-\tfrac12 \eta_{|\kkk(\emptyset)|},\lambda'+\tfrac12 \eta_{|\kkk(\emptyset)|}],
  \end{align*}
  where $(0,0)^K$ is the sequence in $J^K$ containing only $(0,0)$'s. Observe that this does not affect the value of $\proj(S_\kkk^{\lambda'}(0,0))=\proj(S_{\kkk(0,0)^K}^{\lambda'}(0,0))$. By the mean value theorem and Lemma~\ref{lem:trans},
	\begin{equation}\label{eq:firststep}
	  |\proj(S_\lll^{\lambda}(0,0))-\proj(S_{\kkk(\emptyset)}^{\lambda}(0,0))|\leq\delta^{-1}\eta_{|\kkk(\emptyset)|}.
	\end{equation}
  for all $\lambda\in I_\emptyset$.
	
	Let us now define the intervals $I_\omega$ and $\kkk(\omega)$ by induction. If $I$ is an interval, then with notation $cI$ we mean the interval having the same center and $\diam(cI)=c\diam(I)$. Furthermore, for $\omega\in\Omega_*$ let $\omega^-$ be the finite sequence obtained by removing the last element of $\omega$. Let us introduce an auxiliary function $\lll\colon\Omega_* \to \bigcup_{n=1}^\infty J^n$ by setting $\lll(\emptyset)=\lll$ and $\lll(\omega)=\kkk(\omega^-)$ whenever $|\omega|\geq1$. Let $n\geq0$ and suppose that $\kkk(\omega)$ and $I_\omega$ are defined for all $\omega\in\{0,1\}^n$ with the claimed properties. Let $\lambda''\in \delta I_\omega$ be a transcendental number. Choose $m\in\N$ such that
  \begin{align}
    |\lll(\omega)|+m &> \max_{\tau\in\{0,1\}^n}|\kkk(\tau)| \nonumber \\
    \frac{3}{1-\delta}\eta_{|\lll(\omega)|+m} &< \min_{\tau\in\{0,1\}^n}\eta_{|\kkk(\tau)|}, \nonumber \\	
    5\delta^{-1}\eta_{|\lll(\omega)|+m} &< \min\bigl\{|\proj(S_\iii^{\lambda''}(0,0))-\proj(S_\jjj^{\lambda''}(0,0))|: \nonumber \\
    &\qquad\qquad\qquad\qquad\qquad\quad\;\;\proj(S_\iii^{\lambda''}(0,0))\neq\proj(S_\jjj^{\lambda''}(0,0)) \label{eq:other} \\
    &\qquad\qquad\qquad\qquad\qquad\quad\;\,\text{ and } |\iii|,|\jjj|\leq\max_{\tau\in\{0,1\}^n}|\kkk(\tau)|\bigr\} \nonumber.
  \end{align}
  By Lemma~\ref{prop:projint}, $\proj(S_{\lll(\omega)}^{\lambda''}(K_{\lambda''}))$ is an interval with center $\proj(S_{\lll(\omega)}^{\lambda''}(0,0))$. Furthermore, the length of $\proj(S_{\lll(\omega)}^{\lambda''}(K_{\lambda''}))$ is at least $2\lambda^{|\lll(\omega)|}$. By the induction assumption (or, if $\omega=\emptyset$, by \eqref{eq:firststep}) and the fact that $|\kkk(\omega)|\geq K$, we have
  \begin{equation*}
    2\lambda^{|\lll(\omega)|}\geq 2\lambda^{|\kkk(\omega)|}\geq\tfrac32\delta^{-1}\eta_{|\kkk(\omega)|}\geq |\proj(S_{\lll(\omega)}^{\lambda''}(0,0))-\proj(S_{\kkk(\omega)}^{\lambda''}(0,0))|
  \end{equation*}
  and $\proj(S_{\kkk(\omega)}^{\lambda''}(0,0))$ is an interior point of $\proj(S_{\lll(\omega)}^{\lambda''}(K_{\lambda''}))$. Since $\{\proj(S_{\lll(\omega)\lll'}^{\lambda''}(0,0)):|\lll'|\ge m\}$ is dense in $\proj(S_{\lll}^{\lambda''}(K_{\lambda''}))$, there exist $\lll'_1,\lll_2'$ with $|\lll'_1|,|\lll'_2|\ge m$ such that $\lll_1'\neq\lll_2'$ and
  \begin{equation} \label{eq:close}
    |\proj(S_{\kkk(\omega)}^{\lambda''}(0,0))-\proj(S_{\lll(\omega)\lll'_j}^{\lambda''}(0,0))|<\delta\eta_{|\lll(\omega)|+m}
  \end{equation}
  for both $j \in\{1,2\}$. Now, applying Lemma~\ref{lem:const} with $\varepsilon=\delta\eta_{|\lll(\omega)|+m}$ for both $\lll(\omega)\lll_1$ and $\lll(\omega)\lll_2$, we see that there exist disjoint closed intervals $I_{1,j},I_{2,j}$ such that $\diam(I_{k,j})=\eta_{|\lll(\omega)\lll_j|}$,
  \begin{equation} \label{eq:cont2}
  I_{k,j}\subset[\lambda''-3\eta_{|\lll(\omega)|+m},\lambda''+3\eta_{|\lll(\omega)|+m}]
  \end{equation}
  for all $k,j \in \{1,2\}$, and
  \begin{equation} \label{eq:firsthalf}
    \tfrac{1}{2}\delta\eta_{|\lll(\omega)\lll_j'|}<|\proj(S_{\kkk(\omega)}^{\lambda}(0,0))-\proj(S_{\lll(\omega)\lll'_j}^{\lambda}(0,0))|<\tfrac{3}{2}\delta^{-1}\eta_{|\lll(\omega)\lll_j'|}
  \end{equation}
  for all $\lambda\in I_{k,j}$. Since $I_{1,j}\cap I_{2,j}=\emptyset$ for $j\in\{1,2\}$ there exist $k_1,k_2\in\{1,2\}$ such that $I_{k_1,1}\cap I_{k_2,2}=\emptyset$.

  To finish the induction, let us show that the choices $I_{\omega0}=I_{k_1,1}$, $I_{\omega1}=I_{k_2,2}$, $\kkk(\omega0)=\lll(\omega)\lll_1$ and $\kkk(\omega1)=\lll(\omega)\lll_2$ satisfy the claimed properties. The condition \eqref{it:k} follows by the definition. From \eqref{it:k} it follows that
  \begin{equation} \label{eq:bounds}
    0\leq\proj(S_{\kkk(\omega i)}^{\lambda}(0,0))\leq\max\biggl\{\frac{\lambda+\lambda^4+\frac{\lambda^5}{1-\lambda}}{1+\lambda-\lambda^2-\lambda^3-\frac{\lambda^5}{1-\lambda}},\frac{\lambda-\lambda^4+\frac{\lambda^5}{1-\lambda}}{1+\lambda^2+\lambda^3-\frac{\lambda^5}{1-\lambda}}\biggr\}<\frac{\lambda}{1-\lambda}
  \end{equation}
  for all $\lambda\in(\tfrac14,\tfrac13)$, which implies the condition \eqref{it:where}. To show \eqref{it:int1}, it is enough to show that $[\lambda''-3\eta_{|\lll(\omega)|+m},\lambda''+3\eta_{|\lll(\omega)|+m}]\subset I_\omega$. Since $\lambda''\in\delta I_\omega$ and $\diam(I_\omega)=\eta_{|\kkk(\omega)|}$, the inclusion follows from $3\eta_{|\lll(\omega)|+m}+\delta\eta_{|\kkk(\omega)|}<\eta_{|\kkk(\omega)|}$, which is our assumption on $m$. Condition \eqref{it:int2} follows again by the definition of the intervals. The first part of condition \eqref{it:dist} follows by \eqref{eq:firsthalf}. Finally, we prove the second part of \eqref{it:dist}. Let $\iii$ be a finite sequence such that $|\iii|\leq\max_{\tau\in\{0,1\}^n}|\kkk(\tau)|$. If $\proj(S_\iii^{\lambda''}(0,0))=\proj(S_{\kkk(\omega)}^{\lambda''}(0,0))$, then $\proj(S_\iii^{\lambda}(0,0))\equiv\proj(S_{\kkk(\omega)}^{\lambda}(0,0))$. Indeed, since $\lambda''$ is transcendental and  $\lambda\mapsto\proj(S_\iii^{\lambda}(0,0))-\proj(S_{\kkk(\omega)}^{\lambda}(0,0))$ is a ratio of polynomials with integer coefficients, a transcendental $\lambda''$ cannot be a root of it unless the ratio is the constant zero function of $\lambda$. Thus, by \eqref{eq:firsthalf},
  \begin{equation*}
    |\proj(S_{\iii}^{\lambda}(0,0))-\proj(S_{\kkk(\omega i)}^{\lambda}(0,0))|=|\proj(S_{\kkk(\omega)}^{\lambda}(0,0))-\proj(S_{\kkk(\omega i)}^{\lambda}(0,0))|>\tfrac{1}{2}\delta\eta_{|\kkk(\omega i)|}
  \end{equation*}
  for all $\lambda\in I_{\omega i}$.

  Now, let us consider the case when $\proj(S_\iii^{\lambda''}(0,0))\neq \proj(S_{\kkk(\omega)}^{\lambda''}(0,0))$. Clearly, if $\iii$ does not contain any symbol with non-zero second coordinate, then $|\proj(S_\iii^{\lambda''}(0,0))|=\infty$ and so there is nothing to prove. Let $m$ be the difference of the position of the first symbol with non-zero first and the position of the first symbol with non-zero second coordinate. If $m<\frac{-\log3}{\log2}$, then $\lambda^m(1-2\lambda)>\tfrac{2\lambda}{1-\lambda}$ for all $\tfrac14<\lambda<\tfrac13$, and, by \eqref{eq:genbound1} and \eqref{eq:bounds}, clearly
  \begin{equation*}
    |\proj(S_\iii^{\lambda}(0,0))-\proj(S_{\kkk(\omega i)}^\lambda(0,0))|\geq\frac{\lambda}{1-\lambda}>\eta_{|\kkk(\omega i)|}.
  \end{equation*}
  Thus, we may assume that $m>\frac{-\log3}{\log2}$ and so by \eqref{eq:genbound2}, $\bigl|\tfrac{\mathrm{d}}{\mathrm{d}\lambda}\proj(S_\iii^{\lambda}(0,0))\bigr|$ is uniformly bounded by $\delta^{-1}$ on $(\tfrac14,\tfrac13)$ as $\delta^{-1}>144$. Then, by the mean value theorem, \eqref{eq:close}, \eqref{eq:other}, \eqref{eq:cont2}, the monotonicity of $\eta$, and the choice of $\delta$,
  \begin{align*}
    |\proj(S_\iii^{\lambda}(0,0))-&\proj(S_{\kkk(\omega i)}^\lambda(0,0))|\\
    &\geq|\proj(S_\iii^{\lambda''}(0,0))-\proj(S_{\kkk(\omega i)}^{\lambda''}(0,0))|-\delta^{-1}|\lambda-\lambda''|\\
    &\geq|\proj(S_\iii^{\lambda''}(0,0))-\proj(S_{\kkk(\omega)}^{\lambda''}(0,0))|-\delta\eta_{|\lll(\omega)|+m}-\delta^{-1}|\lambda-\lambda''|\\
    &\geq 5\delta^{-1}\eta_{|\lll(\omega)|+m}-\delta\eta_{|\lll(\omega)|+m}-\delta^{-1}|\lambda-\lambda''|\\
    &\geq5\delta^{-1}\eta_{|\lll(\omega)|+m}-\delta\eta_{|\lll(\omega)|+m}-3\delta^{-1}\eta_{|\lll(\omega)|+m}\\
    &\geq (2\delta^{-1}-\delta)\eta_{|\lll(\omega)|+m}\geq\eta_{|\kkk(\omega i)|}
  \end{align*}
  for all $\lambda\in I_{\omega i}$.
\end{proof}

We are now ready to prove the main result.

\begin{proof}[Proof of Theorem~\ref{thm:main2}]
  Let $\{I_\omega\}_{\omega\in\Omega_*}$ be a collection of intervals given by Lemma~\ref{prop:construction} and define $F=\bigcap_{n=0}^\infty\bigcup_{\omega\in\{0,1\}^n}I_\omega$. By the conditions \eqref{it:int1} and \eqref{it:int2} of Lemma~\ref{prop:construction}, it is clear that $F$ is compact, non-empty, and uncountable. Write $\Omega=\{0,1\}^\N$ and define a map $\lambda^*\colon\Omega \to F$ by $\{\lambda^*(\omega)\}=\bigl\{\bigcap_{n=0}^\infty I_{\omega|_n}\bigr\}$, where $\omega|_n$ is the finite sequence obtained by the first $n$ symbols of $\omega$. Clearly, $\lambda^*(\omega)$ is well defined by the conditions \eqref{it:int1} and \eqref{it:int2} of Lemma~\ref{prop:construction}. By Lemma~\ref{prop:construction}\eqref{it:dist} and \eqref{eq:exp}, the sequence $(\proj(S_{\kkk(\omega|_n)}^{\lambda^*(\omega)}(0,0)))_{n \in \N}$ is a Cauchy sequence for every $\omega\in\Omega$. Therefore, we can define $t^*\colon\Omega\to\R$ by setting $t^*(\omega)=\lim_{n\to\infty}\proj(S_{\kkk(\omega|_n)}^{\lambda^*(\omega)}(0,0))$. Note that, by Lemma~\ref{prop:construction}\eqref{it:where}, $t^*(\omega)<\frac{\lambda^*(\omega)}{1-\lambda^*(\omega)}$. Define
  \begin{equation*}
    \Upsilon=\bigcup_{\omega\in\Omega}(\lambda^*(\omega),t^*(\omega))\subset\R^2.
  \end{equation*}
  It is clear that $\Upsilon$ is uncountable (since $F$ is uncountable) and hence, to finish the proof, it suffices to show that $\Upsilon\subset \mathcal{C}_{\eta'}\setminus \mathcal{E}$, where $\eta' = (\eta_n')_{n \in \N}$ is a sequence such that $\eta_n'=\tfrac32\delta^{-1}C\eta_n$ for all $n \in \N$, $\delta$ is as in Lemma~\ref{lem:trans}, and $C$ as in \eqref{eq:exp}.

  Let us first show that $(\lambda^*(\omega),t^*(\omega))\in \mathcal{C}_{\eta'}$ for every $\omega\in\Omega$. Recall that the function $\beta\colon\{1,2,3\}^n\times\{1,2,3\}^n \to J^n$ defined before Lemma \ref{lem:changesys2} is invertible outside the diagonal. Let us define $\beta^{-1}\colon J\to\{1,2,3\} \times \{1,2,3\}$ by $\beta^{-1}(0,0)=(1,1)$ and extend it as before. Let $m(n)$ be the unique integer such that $|\kkk(\omega|_{m(n)-1})|\leq n<|\kkk(\omega|_{m(n)})|$ and define a pair of sequences in $\{1,2,3\}^n$ by $(\iii_n,\jjj_n)=\beta^{-1}\bigl(\kkk(\omega|_{m(n)-1})(0,0)^{n-|\kkk(\omega|_{m(n)-1})|}\bigr)$ for all $n\in\N$. By Lemma~\ref{prop:construction}\eqref{it:k}, as for every $m$ it holds that $\kkk(\omega|_m)|_1=(0,1)$, we have $(\iii_n|_1,\jjj_n|_1)=(2,1)$. Hence,
  \begin{align*}
    |t^*(\omega)-\proj(S_{\beta(\iii_n,\jjj_n)}^{\lambda^*(\omega)}(0,0))|&=|t^*(\omega)-\proj(S_{\kkk(\omega|_{m(n)-1})}^{\lambda^*(\omega)}(0,0))|\\
    &\leq|\lim_{\ell\to\infty}\proj(S_{\kkk(\omega|_{m(\ell)})}^{\lambda^*(\omega)}(0,0))-\proj(S_{\kkk(\omega|_{m(n)-1})}^{\lambda^*(\omega)}(0,0))|\\
    &\leq\sum_{\ell=m(n)}^\infty|\proj(S_{\kkk(\omega|_{\ell})}^{\lambda^*(\omega)}(0,0))-\proj(S_{\kkk(\omega|_{\ell-1})}^{\lambda^*(\omega)}(0,0))|.
  \end{align*}
  By definition, $\lambda^*(\omega)\in I_{\omega|_\ell}$ for every $\ell\geq0$ and by Lemma~\ref{prop:construction}\eqref{it:dist}, we obtain
  \begin{align*}
    \sum_{\ell=m(n)}^\infty|\proj(S_{\kkk(\omega|_{\ell})}^{\lambda^*(\omega)}(0,0))-&\proj(S_{\kkk(\omega|_{\ell-1})}^{\lambda^*(\omega)}(0,0))|\leq\tfrac{3}{2}\delta^{-1}\sum_{\ell=m(n)}^\infty\eta_{|\kkk(\omega|_{\ell})|}\\
    &\leq\tfrac{3}{2}\delta^{-1}\sum_{\ell=|\kkk(\omega|_{m(n)})|}^\infty\eta_\ell\leq\tfrac{3}{2}\delta^{-1}C\eta_{|\kkk(\omega|_{m(n)})|}\leq\tfrac{3}{2}\delta^{-1}C\eta_{n},
  \end{align*}
  where in the last step we applied \eqref{eq:exp}. Recalling Lemma~\ref{lem:changesys2}, we have now shown that $(\lambda^*(\omega),t^*(\omega))\in \mathcal{C}_{\eta'}$.

  Let us then show that $(\lambda^*(\omega),t^*(\omega))\notin \mathcal{E}$. Suppose to the contrary that $(\lambda^*(\omega),t^*(\omega))\in \mathcal{E}$. Then there exists a pair of finite words $\iii,\jjj$ such that $\varphi^{\lambda^*(\omega),t^*(\omega)}_\iii(0)=\varphi^{\lambda^*(\omega),t^*(\omega)}_\jjj(0)$ but $\iii|_1\neq\jjj|_1$ and in particular, $(\iii|_1,\jjj|_1)\in\{(1,2),(2,1)\}$. Without loss of generality, we may assume that $\iii|_1=2$ and $\jjj|_1=1$, and, by possibly extending one of the words by $1$'s, $|\iii|=|\jjj|$. Thus, we have $t^*(\omega)=\proj(S_{\beta(\iii,\jjj)}^{\lambda^*(\omega)}(0,0))$. Choose $n\in\N$ so that $\max_{\tau\in\{0,1\}^n}|\kkk(\tau)|>|\beta(\iii,\jjj)|$ and $\eta_{|\kkk(\omega|_{n+1})|}/\eta_{|\kkk(\omega|_{n+2})|}>3\delta^{-2}C$. Then, by Lemma~\ref{prop:construction}\eqref{it:dist},
  \begin{align*}
    0&=|t^*(\omega)-\proj(S_{\beta(\iii,\jjj)}^{\lambda^*(\omega)}(0,0))|\\&=|\lim_{\ell\to\infty}\proj(S_{\kkk(\omega|_{\ell})}^{\lambda^*(\omega)}(0,0))-\proj(S_{\beta(\iii,\jjj)}^{\lambda^*(\omega)}(0,0))|\\
    &\geq|\proj(S_{\kkk(\omega|_{n+1})}^{\lambda^*(\omega)}(0,0))-\proj(S_{\beta(\iii,\jjj)}^{\lambda^*(\omega)}(0,0))|\\&\qquad\qquad-\sum_{\ell=n+2}^\infty|\proj(S_{\kkk(\omega|_{\ell})}^{\lambda^*(\omega)}(0,0))-\proj(S_{\kkk(\omega|_{\ell-1})}^{\lambda^*(\omega)}(0,0))|\\
    &\geq\tfrac12\delta\eta_{|\kkk(\omega|_{n+1})|}-\tfrac{3}{2}\delta^{-1}C\eta_{|\kkk(\omega|_{n+2})|}>0,
  \end{align*}
   which is a contradiction.
\end{proof}

\section{Average exponential condensation}\label{sec:aec}

In this section, we prove Proposition \ref{thm:main1}. We remark that the proof strongly relies on exact dimensionality proven by Feng and Hu \cite{FengHu2009} and the behavior of the Shannon entropy described by Hochman \cite{Hochman2014}.  More precisely, the proof is standard and essentially follows from Hochman \cite[Section~5.2]{Hochman2014} and Peres and Solomyak \cite{PerSol}. Nevertheless, we give the full details for the convenience of the reader.

Recall that a Borel probability measure $\nu$ on $\R$ is \emph{exact-dimensional} if the lower/upper Hausdorff/packing dimensions of $\nu$ coincide. We refer to the book of Falconer \cite{Falconer1997} for more details on dimensions of measures. Furthermore, the \emph{Shannon entropy} of $\nu$ with respect to the partition $\DD_n = \{[i2^{-n},(i+1)2^{-n})\}_{i\in\mathbb{Z}}$ is
\begin{equation*}
H(\nu,\DD_n)=-\sum_{D\in\DD_n}\nu(D)\log\nu(D)
\end{equation*}
for all $n \in \N$.

\begin{proof}[Proof of Proposition \ref{thm:main1}]
	By Feng and Hu \cite[Theorem 2.8]{FengHu2009}, the natural measure $\mu$ is exact-dimensional. Therefore, by Young \cite[Theorem~4.4]{Young}, it has dimension
	\begin{equation*}
	\dimh(\mu)=\lim_{n\to\infty}\frac{1}{n\log 2}H(\mu,\DD_n).
	\end{equation*}
	Define
	\begin{equation*}
	\mu^{n}=\frac{1}{\# I^n}\sum_{\iii\in I^n}\delta_{\varphi_{\iii}(0)},
	\end{equation*}
	where $\delta_x$ is the Dirac measure at $x$, and let $r(n)$ be the unique integer such that $|\lambda|^{r(n)}\diam(X) \leq 2^{-n} < |\lambda|^{r(n)-1}\diam(X)$. By Hochman \cite[Theorem~1.3]{Hochman2014}, we have
	\begin{equation*}
	\lim_{n\to\infty}\frac{1}{n\log2}\biggl( H(\mu^{r(n)},\DD_{qn}) - H(\mu^{r(n)},\DD_n) \biggr) = 0
	\end{equation*}
	for all $q\in\N$.
	
	Our goal is to show that a closer examination of the Shannon entropy with respect to the partition $\DD_{qn}$ leads us to the claimed formula. Indeed, we shall show that
	\begin{align*}
	\dimh(\mu)&=\lim_{n\to\infty}\frac{1}{n\log 2}H(\mu,\DD_n)=\lim_{n\to\infty}\frac{1}{n\log 2}H(\mu^{r(n)},\DD_n)=\lim_{n\to\infty}\frac{1}{n\log 2}H(\mu^{r(n)},\DD_{qn})\\
	&= -\lim_{n\to\infty}\frac{1}{n\log 2}\int\log\mu^{r(n)}(B(x,2^{-qn})) \dd\mu^{r(n)}(x) = \frac{\log\# I - \Lambda(|\lambda|^q)}{\log|\lambda|^{-1}}
	\end{align*}
	for all $q\in\N$. Note that the first and third equality follow from the above mentioned results.
	
	Let us verify the remaining equalitites. Observe first that we have
	\begin{align*}
	H(&\mu^{r(n)},\DD_{qn}) = -\sum_{D\in\DD_{qn}} \mu^{r(n)}(D)\log\mu^{r(n)}(D) \ge -\int\log\mu^{r(n)}(B(x,2^{-qn})) \dd\mu^{r(n)}(x) \\
	&\ge -\sum_{D\in\DD_{qn}} \mu^{r(n)}(D)\log(\mu^{r(n)}(D-2^{-qn}) + \mu^{r(n)}(D) + \mu^{r(n)}(D+2^{-qn})) \\
	&= -\sum_{D\in\DD_{qn}} \mu^{r(n)}(D)\biggl( \log\mu^{r(n)}(D) + \log\biggl( 1+\frac{\mu^{r(n)}(D-2^{-qn}) + \mu^{r(n)}(D+2^{-qn})}{\mu^{r(n)}(D)} \biggr) \biggr) \\
	&\ge H(\mu^{r(n)},\DD_{qn}) - \sum_{D\in\DD_{qn}} \mu^{r(n)}(D-2^{-qn}) + \mu^{r(n)}(D+2^{-qn}) \\
	&\ge H(\mu^{r(n)},\DD_{qn}) - 2,
	\end{align*}
	which proves the fourth equality. A similar reasoning shows that
	\begin{equation*}
	|H(\mu,\DD_n) - H(\mu^{r(n)},\DD_n)| \le 9
	\end{equation*}
	yielding the second equality. Indeed, this follows since
	\begin{align*}
	\mu(D) &\le \frac{1}{\# I^{r(n)}} \#\{\iii\in I^{r(n)} : \fii_\iii(X) \cap D \ne \emptyset\} \\
	&\le \frac{1}{\# I^{r(n)}} \#\{\iii\in I^{r(n)} : \fii_\iii(0)\in (D-2^{-n}) \cup D \cup (D+2^{-n})\} \\
	&= \mu^{r(n)}(D-2^{-n}) + \mu^{r(n)}(D) + \mu^{r(n)}(D+2^{-n})
	\end{align*}
	and $\mu^{r(n)}(D) \le \mu(D-2^{-n}) + \mu(D) + \mu(D+2^{-n})$
	for all $D \in \DD_{n}$.
	Finally, we also have
	\begin{align*}
	-\int\log\mu^{r(n)}(&B(x,2^{-qn})) \dd\mu^{r(n)}(x) = -\sum_{\iii\in I^{r(n)}} \frac{1}{\# I^{r(n)}} \log\mu^{r(n)}(B(\fii_\iii(0),2^{-qn})) \\
	&= -\sum_{\iii\in I^{r(n)}} \frac{1}{\# I^{r(n)}} \log\frac{\#\{\jjj\in I^{r(n)}:|\fii_\iii(0)-\fii_\jjj(0)|\le2^{-qn}\}}{\# I^{r(n)}} \\
	&= \log\# I^{r(n)} - \sum_{\iii\in I^{r(n)}} \frac{1}{\# I^{r(n)}} \log\#\{\jjj\in I^{r(n)}:|\fii_\iii(0)-\fii_\jjj(0)|\le2^{-qn}\}
	\end{align*}
	which gives the fifth equality and finishes the proof.
\end{proof}

\begin{ack}
  The authors are grateful to Simon Baker, Michael Hochman, and the anonymous referee for helpful comments which improved the presentation.
\end{ack}


\end{document}